\documentclass[reqno,11pt,twoside,english]{amsart}
\usepackage[T1]{fontenc}
\usepackage[latin9]{inputenc}
\usepackage{amsthm}
\usepackage{amssymb}
\usepackage{amsmath,amsfonts,epsfig}

\makeatletter

\numberwithin{equation}{section}

\usepackage{color}
\usepackage[final,linkcolor = blue,citecolor = blue,colorlinks=true]{hyperref}

\usepackage{type1cm,ae}
\usepackage{graphicx}
\usepackage{xspace}
\usepackage{setspace}
\usepackage[english]{babel}
\usepackage{tikz}
\usetikzlibrary{arrows,patterns}
\allowdisplaybreaks[4]

 \newcommand{\rr}{{\mathbb R}}
 \newcommand{\R}{{\mathbb R}}

 \newcommand\loc{{\mathop\mathrm{\,loc\,}}}

 \newcommand\ep{\epsilon}

\def\vint{\mathop{\mathchoice%
		{\setbox0\hbox{$\displaystyle\intop$}\kern 0.22\wd0%
			\vcenter{\hrule width 0.6\wd0}\kern -0.82\wd0}%
		{\setbox0\hbox{$\textstyle\intop$}\kern 0.2\wd0%
			\vcenter{\hrule width 0.6\wd0}\kern -0.8\wd0}%
		{\setbox0\hbox{$\scriptstyle\intop$}\kern 0.2\wd0%
			\vcenter{\hrule width 0.6\wd0}\kern -0.8\wd0}%
		{\setbox0\hbox{$\scriptscriptstyle\intop$}\kern 0.2\wd0%
			\vcenter{\hrule width 0.6\wd0}\kern -0.8\wd0}}%
	\mathopen{}\int}

\numberwithin{equation}{section}
\newtheorem{theorem}{Theorem}[section]

\newtheorem{proposition}[theorem]{Proposition}

\theoremstyle{definition}
\begingroup

\endgroup

\def\XXint#1#2#3{{\setbox0=\hbox{$#1{#2#3}{\int}$}
		\vcenter{\hbox{$#2#3$}}\kern-.5\wd0}}

\makeatother

\usepackage{babel}
\begin{document}

\title[Sharp Morrey regularity for an even order elliptic system]{Sharp Morrey regularity for an even order elliptic system}

\author[C.-Y. Guo and W.-J. Qi]{Chang-Yu Guo and Wen-Juan Qi}

\address[Chang-Yu Guo]{Research Center for Mathematics and Interdisciplinary Sciences, Shandong University 266237,  Qingdao, P. R. China and Frontiers Science Center for Nonlinear Expectations, Ministry of Education, P. R. China}
\email{changyu.guo@sdu.edu.cn}

\address[Wen-Juan Qi]{Research Center for Mathematics and Interdisciplinary Sciences, Shandong University 266237,  Qingdao, P. R. China and Frontiers Science Center for Nonlinear Expectations, Ministry of Education, P. R. China}
\email{wenjuan.qi@mail.sdu.edu.cn}

\thanks{Corresponding author: Wen-Juan Qi.}

\thanks{Both authors are supported by the Young Scientist Program of the Ministry of Science and Technology of China (No. 2021YFA1002200), the National Natural Science Foundation of China (No. 12101362) and the Taishan Scholar Program and the Natural Science Foundation of Shandong Province (No. ZR2022YQ01, ZR2021QA003). }

\begin{abstract}
In this short note, we establish a sharp Morrey regularity theory for an even order elliptic system of Rivi\`ere type:
\begin{equation*}
	\Delta^{m}u=\sum_{l=0}^{m-1}\Delta^{l}\left\langle V_{l},du\right\rangle +\sum_{l=0}^{m-2}\Delta^{l}\delta\left(w_{l}du\right)+f\qquad \text{ in } B^{2m}\label{eq: Longue-Gastel system}
\end{equation*}
under minimal regularity assumptions on the coefficients functions $V_l, w_l$ and that $f$ belongs to certain Morrey space. This can be regarded as a further extension of the recent $L^p$-regularity theory obtained by Guo-Xiang-Zheng \cite{GXZ2022}, and generalizes \cite{Du-Kang-Wang-2022,XZ2023} for second and fourth order elliptic systems.  
\end{abstract}

\maketitle

{\small
\smallskip
\newline
}
\bigskip

\arraycolsep=1pt

\section{Introduction and main results}

In this paper, we consider regularity of weak solutions $u\in W^{m,2}(B^{2m},\rr^n)$ to the following even order linear elliptic system of Rivi\`ere type
\begin{equation}\label{eq:even  order linear elliptic system}
	\Delta^mu=\sum_{l=0}^{m-1}\Delta^l\langle V_l,du\rangle+\sum_{l=0}^{m-2}\Delta^l\delta(w_ldu)+f\qquad\text{in\,\,}B^{2m},
\end{equation}
where $B^{2m}=B^{2m}_1(0)\subset\rr^{2m}$ is the unit ball, $f\in M^{p,p\lambda}(B^{2m},\R^n)$ with $p\ge 1$ and the coefficient functions satisfy 
\begin{equation}
	\begin{split}
		&w_l\in W^{2l+2-m,2}(B^{2m},\rr^{n\times n})\qquad\text{for\,\,}l\in\{0,...,m-2\}\\
		&V_l\in W^{2l+1-m,2}(B^{2m},\rr^{n\times n}\otimes\wedge^1\rr^{2m})\qquad\text{for\,\,}l\in\{0,...,m-1\}.
	\end{split}
\end{equation}
Moreover, the first order potential $V_0$ has the decomposition $V_0=d\eta+F$ with
\begin{equation}
	\eta\in W^{2-m,2}(B^{2m},so(n)),\quad F\in W^{2-m,\frac{2m}{m+1},1}(B^{2m},\rr^{n\times n}\otimes\wedge^1\rr^{2m}).
\end{equation}
Here $so(n)$ represents the space of $n\times n$ antisymmetric matries. 

The homogeneous version of \eqref{eq:even  order linear elliptic system}, that is when $f\equiv0$, was first introduced by Rivi\`ere \cite{R2007} for the case $m=1$, and later by Lamm-Rivi\`ere \cite{Lamm-Riviere-2008} for the case $m=2$, and finally by de Longueville and Gastel \cite{DG2021} for the general case. It includes many interesting geometric models, such as the harmonic mapping equations ($m=1$), the prescribed mean curvature equations ($m=1$), the biharmonic mapping equations ($m=2$), the $m$-polyharmonic mapping equations and so on; we refer the interested readers to \cite{R2007,Riviere-Struve-2008,Chang-W-Y-1999,Struwe-2008,Wang-2004-CPAM,Goldstein-Strzelecki-Zatorska-2009,GS2009,Lamm-Wang-2009,Moser-2015-TAMS,AY17,DG2021,Horter-Lamm-2021,Guo-Xiang-2021,GXZ2021,GXZ2022,Chen-Zhu-2023,HJL2023} and the references therein for various aspects regarding this system or polyharmonic mappings. 

To explain the difficulty toward regularity issues, let us look at the simplest case when $m=1$. In this case, system \eqref{eq:even  order linear elliptic system} with $f\equiv 0$ reduces to a second order elliptic PDE
\begin{equation}\label{eq:2 dim riviere system}
	-\Delta u=\Omega\cdot\nabla u\qquad\text{in\,\,}B^2,
\end{equation}
which was initially introduced by Rivi\`ere in his celebrated work \cite{R2007}. The right-hand side of \eqref{eq:2 dim riviere system} lies merely in $L^1$ and thus prevents a direct application of the standard iteration techniques from regularity theory of elliptic equations. A fundamental observation, due to Rivi\`ere \cite{R2007}, was the algebraic anti-symmetry of $\Omega$, allows people to find a conservation law of \eqref{eq:2 dim riviere system}, turning \eqref{eq:2 dim riviere system} into a divergence form. Starting from the conservation law, continuity and compactness follow routinely via standard analytic tools.

Rivi\`ere's conservation law approach was soon extended to higher order systems in \cite{Lamm-Riviere-2008} ($m=2$) and \cite{DG2021} ($m\geq 3$). It should be noticed, however, that direct extension of this approach only gives continuity of weak solutions. A refined approach for H\"older regularity of \eqref{eq:even  order linear elliptic system} (for $f\equiv 0$) can be found in \cite{Guo-Xiang-2019-Boundary,Guo-Xiang-2021}. 

Geometric applications, such as the heat flow or bubbling analysis of polyharmonic mappings, motivate people to establish a deeper $L^p$-regularity of weak solutions to \eqref{eq:even  order linear elliptic system}; see \cite{ST2013,Moser-2015-TAMS,AY17,Chen-Zhu-2023,GXZ2022} for more on the motivation/applications. The case $m=1$ was studied by Sharp-Topping \cite{ST2013}, the case $m=2$ was considered by Guo-Xiang-Zheng \cite{GXZ2021} (see also \cite{Guo-Wang-Xiang-2022-CV} for the supercritical case), and the general case $m\geq 3$ was established very recently by Guo-Xiang-Zheng \cite{GXZ2022}. The starting point of \cite{GXZ2022} is the following conservation law of  de Longueville and Gastel \cite{DG2021} (see also \cite{Horter-Lamm-2021} for another version of conservation law). To describe their convservation law, for $D\subset \R^{2m}$, we set
\begin{equation}\label{eq:theta for small coefficient}
	\begin{aligned}
		\theta_{D}:=\sum_{k=0}^{m-2}&\|w_k\|_{W^{2k+2-m,2}(D)}+\sum_{k=1}^{m-1}\|V_k\|_{W^{2k+1-m,2}(D)}\\
		&+\|\eta\|_{W^{2-m,2}(D)}+\|F\|_{W^{2-m,\frac{2m}{m+1},1}(D)}
	\end{aligned}
\end{equation}
Under a smallness assumption
\begin{equation}\label{eq:smallness assumption}
	\theta_{B^{2m}_1}<\ep_m,
\end{equation}
they successfully found 
$$A\in W^{m,2}\cap L^\infty(B^{2m},Gl(n)) \text{ and } B\in W^{2-m,2}(B^{2m},\R^{n\times n}\otimes \wedge^2\R^{2m})$$  which satisfies
\begin{equation*}
	\Delta^{m-1}dA+\sum_{k=0}^{m-1}(\Delta^k A)V_k-\sum_{k=0}^{m-2}(\Delta^k dA)w_k=\delta B,
\end{equation*}
such that $u\in W^{m,2}(B^m,\R^n)$ solves \eqref{eq:even  order linear elliptic system} in $B^{2m}$\footnote{In the orginal paper \cite{DG2021}, the authors only obtained this on $B^{2m}_{1/2}$. But a minor change of arguments lead to the current form, for details, see \cite{GXZ2023}} if and only if it satisfies the following conservation law (namely, in divergence form): 
\newline 
\begin{equation}\label{eq:conservation law 1}
	\begin{aligned}
		0&=\delta\Big[\sum_{l=0}^{m-1}\left(\Delta^{l} A\right) \Delta^{m-l-1} d u-\sum_{l=0}^{m-2}\left(d \Delta^{l} A\right) \Delta^{m-l-1} u \\ &\qquad -\sum_{k=0}^{m-1} \sum_{l=0}^{k-1}\left(\Delta^{l} A\right) \Delta^{k-l-1} d\left\langle V_{k}, d u\right\rangle+\sum_{k=0}^{m-1} \sum_{l=0}^{k-1}\left(d \Delta^{l} A\right) \Delta^{k-l-1}\left\langle V_{k}, d u\right\rangle \\ &\qquad -\sum_{k=0}^{m-2} \sum_{l=0}^{k-2}\left(\Delta^{l} A\right) d \Delta^{k-l-1} \delta\left(w_{k} d u\right)+\sum_{k=0}^{m-2} \sum_{l=0}^{k-2}\left(d \Delta^{l} A\right) \Delta^{k-l-1} \delta\left(w_{k} d u\right) \\ &\qquad -\langle B, d u\rangle\Big]+Af,
	\end{aligned}
\end{equation}
where $d \Delta^{-1} \delta$  denotes the identity map. The general idea of \cite{GXZ2022} is similar to Sharp-Topping \cite{ST2013} but with principal technical differences. One crucial step is to establish a suitable Morrey decay for all the gradients of weak solutions when $f\in L^p$ for $p>1$. Since $L^p\subset M^{1,\lambda}$ for some $\lambda\in (0,1)$, it is natrual to weaken the assumption $f\in L^p$ to $f\in M^{1,\lambda}$. Our first main result gives the optimal Morrey decay when $f\in M^{1,\lambda}$.  

\begin{theorem}\label{thm:holder continuity}
	Let $u\in W^{m,2}(B^{2m},\rr^n)$ be a weak solution of \eqref{eq:even  order linear elliptic system} with $f\in M^{1,\lambda}(B^{2m},\R^n)$ for some $0<\lambda<1$. Then 
	$$\nabla^iu\in M^{\frac{2m}{i},\frac{2m\lambda}{i}}_{\text{loc}}(B^{2m})\qquad1\leq i\leq m,$$ with
	\begin{equation}\label{eq:decay estimate}
		\sum_{i=1}^{m}\|\nabla^iu\|_{M^{2m/i,2m\lambda/i}(B^{2m}_{1/2})}\leq C\left(\|u\|_{W^{m,2}(B^{2m}_1)}+\|f\|_{M^{1,\lambda}(B^{2m}_1)}\right).
	\end{equation}
As a result, we have $u\in C^{0,\lambda}_{\text{loc}}(B^{2m})$ with 
\begin{equation}\label{eq:holder estimate 1}
	\|u\|_{C^{0,\lambda}(B^{2m}_{1/2})}\leq C \left(\|u\|_{W^{m,2}(B^{2m}_1)}+\|f\|_{M^{1,\lambda}(B^{2m}_1)}\right),
\end{equation}
where $C>0$ is a positive constant depending only on $m,n,\lambda$ and the coefficient functions $V_l,w_l$.
\end{theorem}

Theorem \ref{thm:holder continuity} reduces to the main result of Du-Kang-Wang \cite{Du-Kang-Wang-2022} when $m=1$, and to the main result of Xiang-Zheng \cite{XZ2023} when $m=2$. The sharpness of Theorem \ref{thm:holder continuity} can be seen in the following way. Consider the model case $\Delta^m u=f\in L^p(B^{2m})$ with $1<p<\frac{2m}{2m-1}$. Then $f\in L^p\subset M^{1,\lambda}$ with $\lambda=2m(1-1/p)\in (0,1)$. On the other hand, $u\in W^{2m,p}_{\loc}\subset C^{0,\lambda}_{\loc}$, which shows the Morrey regularity of $\nabla u$ is optimal. Notice also that Theorem \ref{thm:holder continuity} fails for the case $\lambda=1$; see \cite{GXZ2022} for a non-Lipschitz continuous solutions.

Our second main result shows that weak solutions of \eqref{eq:even  order linear elliptic system} enjoy higher regularity if $f$ has higher Morrey regularity. 
\begin{theorem}\label{thm:morrey estimate}
		Let $u\in W^{m,2}(B^{2m},\rr^n)$ be a weak solution of \eqref{eq:even  order linear elliptic system} with $f\in M^{p,p\lambda}(B^{2m})$ for some $1<p<\frac{2m}{2m-1}$ and $0\leq\lambda<\frac{2m-(2m-1)p}{p}$. Then 
		$$\nabla^iu\in M^{p_i,p_i\lambda}_{\text{loc}}(B^{2m})\qquad 1\leq i\leq m+1,$$
		where $p_i=\frac{2mp}{2m-(2m-i)p}$ and
		\begin{equation}\label{eq:morrey estimate}
			\sum_{i=1}^{m+1}\|\nabla^iu\|_{M^{p_i,p_i\lambda}(B^{2m}_{1/2})}\leq C\left(\|u\|_{W^{m,2}(B^{2m}_1)}+\|f\|_{M^{p,p\lambda}(B^{2m}_1)}\right)
		\end{equation}
	holds for some constant $C>0$ depending only on $m,n,p,\lambda$ and the coefficient functions $V_l,w_l$.
\end{theorem}
If $\lambda=0$, then Theorem \ref{thm:morrey estimate} reduces to the main $L^p$-regularity theorem of Guo-Xiang-Zheng \cite{GXZ2022}. As was observed in \cite{GXZ2022}, one cannot expect higher (Morrey-)Sobolev regularity for $\nabla^iu$ with $i>m+1$ and thus Theorem \ref{thm:morrey estimate} is optimal in this sense. 

As the general idea for the proofs of Theorem \ref{thm:holder continuity} and Theorem \ref{thm:morrey estimate} is very similar to that used in Guo-Xiang-Zheng \cite{GXZ2022}. We will follow closely the presentation in \cite{GXZ2022} and indicate the necessary changes when necessary. 

\section{Morrey spaces and fractional Riesz operators}
Let $\Omega\subset\rr^n$ be a bounded smooth domain, $1\leq p<\infty$ and $0\leq\lambda<n$. The Morrey space $M^{p,\lambda}(\Omega)$ consists of function $f\in L^p(\Omega)$ such that
\begin{equation*}
	\|f\|_{M^{p,\lambda}(\Omega)}\equiv\left(\sup_{x\in\Omega,r>0}r^{-\lambda}\int_{B_r(x)\cap\Omega}|f|^p\right)^{1/p}<\infty.
\end{equation*}
Denote by $L^p_*$ the weak $L^p$ space and define the weak Morrey space $M_*^{p,\lambda}(\Omega)$ as the space of functions $f\in L^p_*(\Omega)$ such that 
\begin{equation*}
	\|f\|_{M_*^{p,\lambda}(\Omega)}\equiv\left(\sup_{x\in\Omega,r>0}r^{-\lambda}\|f\|^p_{L^p_*(B_r(x)\cap\Omega)}\right)^{1/p}<\infty,
\end{equation*}
where 
\begin{equation*}
	\|f\|^p_{L^p_*(B_r(x)\cap\Omega)}\equiv\sup_{t>0}t^p\left|\{x\in B_r(x)\cap\Omega:|f(x)|>t\}\right|.
\end{equation*}

Let $0<\alpha<n$. Then the Riesz operators $I_\alpha f$ of a locally integrable function $f$ on $\rr^n$ is the function defined by
$$I_\alpha(f)(x):=\frac{1}{c_{\alpha,n}}\int_{\rr^n}|x-y|^{\alpha-n}f(y)\,dy,$$
where the constant $C_{\alpha,n}$ is given by
$$c_{\alpha,n}=\pi^{n/2}2^\alpha\frac{\Gamma(\alpha/2)}{\Gamma((n-\alpha)/2)}.$$

The following well-known estimates on fractional Riesz operators between Morrey spaces were proved by Adams \cite{A1975}.

\begin{proposition}\label{prop:Adams 1975} Let $0< \alpha<n$, $0\le \lambda< n$ and $1\le p<\frac{n-\lambda}{\alpha}$. There exists a constant $C>0$ depending only $n,\alpha,\lambda$ and $p$ such that, for all $f\in M^{p,\lambda}(\R^n)$, there holds
	
	{\upshape(i)} If  $p>1$, then
	\begin{equation}\label{eq:Riesz Adams 1}
		\|I_\alpha(f)\|_{M^{\frac{(n-\lambda)p}{n-\lambda-\alpha p},\lambda}(\R^n)}\leq C\|f\|_{M^{p,\lambda}(\R^n)}.
	\end{equation}
	
	{\upshape(ii)} If $p=1$, then
	\begin{equation}\label{eq:Riesz Adams 2}
		\|I_\alpha(f)\|_{M_{*}^{\frac{n-\lambda}{n-\lambda-\alpha},\lambda}(\R^n)}\leq C\|f\|_{M^{1,\lambda}(\R^n)}.
	\end{equation}	
\end{proposition}
%
In particular, when $\lambda=0$, it reduces to the boundedness theory of Riesz operator between $L^p$ spaces.

\section{H\"older continuity via decay estimates}
This section is devoted to the proof of Theorem~\ref{thm:holder continuity}. Following \cite[Proof of Lemma 3.2]{GXZ2022}, we divide the proof into four steps. 
\begin{proof}[Proof of Theorem \ref{thm:holder continuity}]
First of all, since the result is local and scaling invariant, we may assume that for a sufficiently small $\varepsilon$, the conservation law \eqref{eq:conservation law 1} holds for some $A,B$ in $B^{2m}_{1/2}$; for details, see \cite{GXZ2022}.

\textbf{Step 1.} Rewrite the equation. 

According to \cite[Proposition 3.3]{GXZ2022}, $Adu$ satisfies the equation
\begin{equation}
	\delta\Delta^{m-1}(Adu)=\sum_{i=1}^{m-1}\delta^i\left(\sum_{j=m-i}^{m}\nabla^jA\nabla^{2m-i-j}u\right)+\delta K+Af,
\end{equation}
where $\delta$ denotes the divergence operator, $\delta^i$ means taking divergence for $i$ times and $K$ is the last five terms of the conservation law~\eqref{eq:conservation law 1}.

\textbf{Step 2.} Decompose $Adu$.

Extend all the functions from $B_{1/2}$ into the whole space $\rr^{2m}$ in a bounded way and for simplicity use the same notations to denote the extended functions. For $f$, simply extends it as zero outside $B_{1/2}$. As in \cite[Page 299]{GXZ2022}, let $c\log|\cdot|$ be the fundamental solution of $\Delta^m$ in $\rr^{2m}$ and define
$$u_{11}=c\log*\left(\sum_{i=1}^{m-1}\delta^i\left(\sum_{j=m-i}^{m}\nabla^jA\nabla^{2m-i-j}u\right)+\delta K\right),\quad u_{12}=c\log*(Af)$$
and
$$u_2=c\log*\Delta^{m-1}(dA\wedge du).$$
Then we have $\Delta^mu_{11}=\sum_{i=1}^{m-1}\delta^i\left(\sum_{j=m-i}^{m}\nabla^jA\nabla^{2m-i-j}u\right)+\delta K$, $\Delta^mu_{12}=Af$ and $\Delta^mu_2=\Delta^{m-1}(dA\wedge du)$. Thus, we obtain the decomposition
$$Adu=du_{11}+du_{12}+d^*u_2+h$$
for some $m$-polyharmonic $1$-form in $B^{2m}_{1/2}$.

\textbf{Step 3.} Estimates of $u_{11},u_{12},u_2$ and $h$.

For the  term $u_{11}$ and $u_2$, we may repeat the proof in \cite[Step 3 in the proof of Lemma 3.2]{GXZ2022} (even simpler, using the boundedness of Riesz operator on $L^p$ spaces, instead of Lorentz spaces) to obtain
\begin{equation*}
	\sum_{j=1}^{m}\|\nabla^ju_{11}\|_{L^{2m/j}(\rr^{2m})}\lesssim\varepsilon\sum_{i=1}^{m}\|\nabla^iu\|_{L^{2m/i}(B^{2m}_{1/2})}
\end{equation*}
and
\begin{equation*}
	\sum_{j=1}^{m}\|\nabla^ju_2\|_{L^{2m/j}(\rr^{2m})}\lesssim\varepsilon\sum_{i=1}^{m}\|\nabla^iu\|_{L^{2m/i}(B^{2m}_{1/2})}.
\end{equation*}
For $u_{12}$, we use the estimate
$$|\nabla^ju_{12}|\lesssim I_{2m-j}(|Af|)\qquad\text{for all\,\,}1\leq j\leq 2m$$
and the boundedness of the operator
$$I_{2m-j}\colon M^{1,\lambda}(\rr^{2m})\to M_*^{\frac{2m-\lambda}{j-\lambda},\lambda}(\rr^{2m})$$
and find that $\nabla^ju_{12}\in M_*^{\frac{2m-\lambda}{j-\lambda},\lambda}(\rr^{2m})$ with estimate
$$\|\nabla^ju_{12}\|_{M_*^{\frac{2m-\lambda}{j-\lambda},\lambda}(\rr^{2m})}\lesssim\|f\|_{M^{1,\lambda}(\rr^{2m})}\lesssim\|f\|_{M^{1,\lambda}(B^{2m}_{1/2})}.$$
For any $r>0$, it follows from H\"older's inequality that
$$\|\nabla^ju_{12}\|_{L^{2m/j}(B^{2m}_r)}
\leq\|\nabla^ju_{12}\|_{L_*^{\frac{2m-\lambda}{j-\lambda}}(B^{2m}_r)}r^{2m\left(\frac{j}{2m}-\frac{j-\lambda}{2m-\lambda}\right)}
\leq C\|f\|_{M^{1,\lambda}(B^{2m}_{1/2})}r^\lambda.$$
Summing over $j$, we infer that 
$$\sum_{j=1}^{m}\|\nabla^ju_{12}\|_{L^{2m/j}(B^{2m}_r)}\lesssim  \|f\|_{M^{1,\lambda}(B^{2m}_{1/2})}r^\lambda.$$
For the polyharmonic function $h$, it follows from \cite[Lemma 6.2]{GS2009} that, for any $0<r<\frac{1}{4}$,
$$\sum_{j=1}^{m}\|\nabla^jh\|_{L^{2m/j}(B^{2m}_r)}\lesssim  r\sum_{i=1}^{m}\|\nabla^ih\|_{L^{2m/i}(B^{2m}_{1/2})}.$$

\textbf{Step 4.} Conclusion.

For any $0<r<\frac{1}{4}$ and $1\leq j\leq m$, the triangle inequality implies that
\begin{equation*}
	\begin{split}
		\|\nabla^ju\|_{L^{2m/j}(B^{2m}_r)}
		&\leq\|\nabla^{j-1}(A^{-1}h)\|_{L^{2m/j}(B^{2m}_r)}+\|\nabla^{j-1}(A^{-1}du_{11})\|_{L^{2m/j}(B^{2m}_r)}\\
		&\quad+\|\nabla^{j-1}(A^{-1}du_{12})\|_{L^{2m/j}(B^{2m}_r)}+\|\nabla^{j-1}(A^{-1}*du_2)\|_{L^{2m/j}(B^{2m}_r)}\\
		&\lesssim r\sum_{i=1}^{m}\|\nabla^ih\|_{L^{2m/i}(B^{2m}_{1/2})}+\varepsilon\sum_{i=1}^{m}\|\nabla^iu\|_{L^{2m/i}(B^{2m}_{1/2})}+\|f\|_{M^{1,\lambda}(B^{2m}_{1/2})}r^\lambda\\
		&\leq C(r+\varepsilon)\sum_{i=1}^{m}\|\nabla^iu\|_{L^{2m/i}(B^{2m}_{1/2})}+C\|f\|_{M^{1,\lambda}(B^{2m}_{1/2})}r^\lambda.
	\end{split}
\end{equation*}
Summing over $j$, we obtain
$$\sum_{i=1}^{m}\|\nabla^iu\|_{L^{2m/i}(B^{2m}_r)}\leq C(r+\varepsilon)\sum_{i=1}^{m}\|\nabla^iu\|_{L^{2m/i}(B^{2m}_{1/2})}+C\|f\|_{M^{1,\lambda}(B^{2m}_{1/2})}r^\lambda,$$
where $C$ is a constant depending only $m,n$ and $\lambda$. 
Write 
$$\Theta(r)=\sum_{i=1}^{m}\|\nabla^iu\|_{L^{2m/i}(B^{2m}_r)}.$$
Then for any $0<r<\frac{1}{4}$
$$\Theta(r)\leq C(r+\varepsilon)\Theta(\frac{1}{2})+C\|f\|_{M^{1,\lambda}(B^{2m}_{1/2})}r^\lambda$$
for some $C=C(m,n,\lambda)>0$. Now choose $r=\tau$ small such that $2C\tau\leq\tau^{(\lambda+1)/2}$, and then choose $\varepsilon\leq \tau$, we obtain the decay estimate
$$\Theta(\tau)\leq\tau^{\frac{\lambda+1}{2}}\Theta(\frac{1}{2})+C\|f\|_{M^{1,\lambda}(B^{2m}_{1/2})}\tau^\lambda.$$
Now using a standard scaling~\cite[Section 2.3]{GXZ2022} and iteration argument (see \cite[Proof of Theorem 3.1]{GXZ2021}), we conclude that for any $k\geq 1$,
$$\Theta(\tau^k)\leq\tau^{\frac{\lambda+1}{2}}\Theta(\tau^{k-1})+C\|f\|_{M^{1,\lambda}(B^{2m}_{1/2})}\tau^{k\lambda},$$
which implies that
$$\Theta(r)\leq Cr^\lambda\left(\Theta(1)+\|f\|_{M^{1,\lambda}(B^{2m}_{1})}\right)$$
for all $0<r<\frac{1}{4}$. This gives $\nabla^iu\in M^{\frac{2m}{i},\frac{2m\lambda}{i}}_{\text{loc}}(B^{2m})$ for all $1\leq i\leq m$ together with the desired estimate~\eqref{eq:decay estimate}.

Finally, Morrey's Dirichlet growth theorem (see  e.g.~\cite{G1983}) implies that $u\in C^{0,\lambda}_{\text{loc}}(B^{2m})$ and the H\"older continuity estimate~\eqref{eq:holder estimate 1}. This completes the proof of Theorem~\ref{thm:holder continuity}.   
\end{proof}

\section{Optimal local estimates}
In this section, we shall prove Theorem~\ref{thm:morrey estimate}. The general strategy is very similar to \cite[Proof of Theorem 1.2]{GXZ2022}. Before the official proofs, we shall point out three easy consequences of Theorem~\ref{thm:holder continuity}. 

1) By H\"older's inequality, we have
	\begin{equation*}
		M^{p,p\lambda}(B^{2m}_1)\subset M^{1,\lambda_0}(B^{2m}_1),
	\end{equation*}
where
\begin{equation}\label{eq:lambda_0}
	\lambda_0=\lambda+2m(1-\frac{1}{p}).
\end{equation}
Thus Theorem~\ref{thm:holder continuity} implies that
$$\sum_{i=1}^{m}\|\nabla^iu\|_{M^{2m/i,2m\lambda_0/i}(B^{2m}_{1/2})}\leq C\left(\|u\|_{W^{m,2}(B^{2m}_1)}+\|f\|_{M^{1,\lambda_0}(B^{2m}_1)}\right).$$
and that $u\in C^{0,\lambda_0}(B^{2m}_{1/2})$ with 
\begin{equation}\label{eq:holder estimate}
	\|u\|_{C^{0,\lambda_0}(B^{2m}_{1/2})}\leq C \left(\|u\|_{W^{m,2}(B^{2m}_1)}+\|f\|_{M^{1,\lambda_0}(B^{2m}_1)}\right).
\end{equation}

2) Since $M^{p,p\lambda}(B^{2m}_1)\subset L^p(B^{2m}_1)$ with $1<p<\frac{2m}{2m-1}$, \cite[Theorem 1.2]{GXZ2022} implies that $u\in W^{m+1,p_{m+1}}(B^{2m}_{1/2})$ together with the estimate
\begin{equation*}
	\|u\|_{W^{m+1,p_{m+1}}(B^{2m}_{1/2})}\leq C\left(\|u\|_{W^{m,2}(B^{2m}_1)}+\|f\|_{L^p(B^{2m}_1)}\right),
\end{equation*}
where $p_{m+1}=\frac{2mp}{2m-(m-1)p}.$

3) Thanks to the above $W^{m+1,p_{m+1}}$-regularity, we can then repeat the proof of Theorem \ref{thm:holder continuity} to deduce that
\begin{equation*}
	\|\nabla^{m+1}u\|_{M^{\frac{2m}{m+1},\frac{2m\lambda_0}{m+1}}{(B^{2m}_{1/2}})}\leq C \left(\|u\|_{W^{m,2}(B^{2m}_1)}+\|f\|_{M^{1,\lambda_0}(B^{2m}_1)}\right).
\end{equation*}

Therefore, we summarize the above result to conclude that
\begin{equation*}
	\sum_{i=1}^{m+1}\|\nabla^iu\|_{M^{2m/i,2m\lambda_0/i}(B^{2m}_{1/2})}\leq CM, 
\end{equation*}
where
\begin{equation}\label{eq:M}
	M=\left(\|u\|_{W^{m,2}(B^{2m}_1)}+\|f\|_{M^{p,p\lambda}(B^{2m}_1)}\right).
\end{equation}

With all the previous ingredients at hand, we are now ready to prove Theorem~\ref{thm:morrey estimate}.

\begin{proof}[Proof of Theorem~\ref{thm:morrey estimate}]
Module some technical arguments, the proof presented here is very similar to \cite[Section 5]{GXZ2022}. By the discussion above, we know that $u\in W^{m+1,p_{m+1}}(B^{2m}_{1/2})$. As in \cite[Proof of Theorem 1.2]{GXZ2022}, we consider separately two cases.
\medskip 

\textbf{Case I}: $m$ is an even integer.
\medskip

According to \cite[Corollary 3.4]{GXZ2022}, $A\Delta u$ satisfies the equation
\begin{equation}
	\begin{split}
			\Delta^{m-1}(A\Delta u)&=\sum_{i=1}^{m-1}\delta^i\left(\sum_{j=m-i}^{m}\nabla^jA\nabla^{2m-i-j}u\right)-\Delta^{m-1}(dAdu)+\delta K+Af\\
			&=\sum_{i=1}^{m-1}\delta^i\left(\sum_{j=m-i}^{m}\nabla^jA\nabla^{2m-i-j}u\right)+\delta K+Af,
	\end{split}
\end{equation}
where $\delta$ denotes the divergence operator, $\delta^i$ means taking divergence for $i$ times and $K$ is the last five terms of the conservation law~\eqref{eq:conservation law 1}.  Here, we use $\sum_{i}a_i$ to denote a linear combination of $a_i$'s, i.e., $\sum_ia_i=\sum_ic_ia_i$ for some harmless absolute constant $c_i$.
 
\textbf{Step 1.} Split $A\Delta u$.
 
 Fix $x_0\in B^{2m}_{1/4}$ and $0<r<\frac{1}{4}$. Split $A\Delta u=v+h$ in $B_r(x_0)$ with $v$ and $h$ satisfying
 \begin{equation*}
 	\begin{cases}
 		\Delta^{m-1}h=0\qquad\quad\,\text{in\,\,}B_r(x_0),\\
 		\Delta^ih=\Delta^i(A\Delta u)\quad\text{on\,\,}\partial B_r(x_0),\quad 0\leq i\leq m-2
 	\end{cases}
 \end{equation*} 
and
\begin{equation*}
	\begin{cases}
		\Delta^{m-1}v=\Delta^{m-1}(A\Delta u)\qquad\qquad\,\text{in\,\,}B_r(x_0),\\
		v=\Delta^{\frac{m}{2}}v=\cdots=\Delta^{m-2}v=0\quad\,\text{on\,\,}\partial B_r(x_0).
	\end{cases}
\end{equation*}
Then $v$ satisfies 
\begin{equation*}
		\Delta^{m-1}v=\sum_{i=1}^{m-1}\delta^i\left(\sum_{j=m-i}^{m}\nabla^jA\nabla^{2m-i-j}u\right)+\delta K+Af.
\end{equation*}

\textbf{Step 2.} A duality argument. 

In this step, we will divide it into two parts.

\textbf{Part 1.} Note that $p_m=\frac{2p}{2-p}$. By the duality of $L^p$-norm, we have  $$\|\Delta^{\frac{m-2}{2}}v\|_{L^{p_m}(B_r(x_0))}=\sup_{\varphi\in\mathcal{A}_1}\int_{B_r(x_0)}(\Delta^{\frac{m-2}{2}}v)\varphi,$$
where
$$\mathcal{A}_1=\{\varphi\in L^{p_m'}(B_r(x_0),\rr^m)\colon\|\varphi\|_{L^{p_m'}(B_r(x_0))}\leq 1\}$$
and $p_m'=\frac{p_m}{p_m-1}$ is the H\"older conjugate exponent of $p_m$. For any $\varphi\in\mathcal{A}_1$, let $\Phi\in W^{\frac{m}{2},p_m'}(B_r(x_0))$ satisfy
\begin{equation*}
	 	\begin{cases}
		\Delta^{\frac{m}{2}}\Phi=\varphi\qquad\qquad\qquad\qquad\quad\,\,\,\text{in\,\,}B_r(x_0),\\
        \Phi=\Delta\Phi=\cdots=\Delta^{\frac{m}{2}-1}\Phi=0\quad\,\text{on\,\,}\partial B_r(x_0).
	\end{cases}
\end{equation*}
By the standard elliptic regularity theory, there exists a constant $C_p>0$ such that
\begin{equation}\label{eq:Phi estimate}
	\|\Phi\|_{W^{m,p_m'}(B_r(x_0))}\leq C_p\|\varphi\|_{L^{p_m'}(B_r(x_0))}\leq C_p.
\end{equation}
Note that integration by parts gives
$$\int_{B_r(x_0)}(\Delta^{\frac{m-2}{2}}v)\varphi=\int_{B_r(x_0)}(\Delta^{\frac{m-2}{2}}v)(\Delta^{\frac{m}{2}}\Phi)=\int_{B_r(x_0)}(\Delta^{m-1}v)\Phi.$$
For the details, see \cite[Page 315 - 316]{GXZ2022}. Thus, we have
$$\sup_{\varphi\in\mathcal{A}_1}\int_{B_r(x_0)}(\Delta^{\frac{m-2}{2}}v)\varphi
\leq C\sup_{\varphi\in\mathcal{A}_2}\int_{B_r(x_0)}(\Delta^{\frac{m-2}{2}}v)(\Delta^{\frac{m}{2}}\Phi)
=C\sup_{\varphi\in\mathcal{A}_2}\int_{B_r(x_0)}(\Delta^{m-1}v)\Phi,$$
where
$$\mathcal{A}_2=\left\{\Phi\in W^{m,p_m'}(B_r(x_0))\colon\|\Phi\|_{W^{m,p_m'}}\leq 1,\Phi=\Delta\Phi=\cdots=\Delta^{\frac{m}{2}-1}\Phi=0\text{\,\,on\,\,}\partial B_r(x_0)\right\}.$$
Recall that 
$$\int_{B_r(x_0)}(\Delta^{m-1}v)\Phi=\int_{B_r(x_0)}\left\{\sum_{i=1}^{m-1}\delta^i\left(\sum_{j=m-i}^{m}\nabla^jA\nabla^{2m-i-j}u\right)+\delta K+Af\right\}\Phi.$$
Applying H\"older's inequality and the Sobolev embedding $W^{m,p_m'}_0(B_r(x_0))\subset L^{p'}(B_r(x_0))$, we  obtain
\begin{equation*}
	\begin{split}
		\int_{B_r(x_0)}(Af)\Phi
		&\leq\|Af\|_{L^p(B_r(x_0))}\|\Phi\|_{L^{p'}(B_r(x_0))}
		\leq C\|f\|_{L^p(B_r(x_0))}\|\Phi\|_{W_0^{m,p_m'}(B_r(x_0))}\\
		&\leq C_p\|f\|_{L^p(B_r(x_0))}
		\leq C_p\|f\|_{M^{p,p\lambda}(B^{2m})}r^\lambda.
	\end{split}
\end{equation*}
For the remaining terms, we can follow \cite[Step 2 on Page 316]{GXZ2022} to obtain 
$$\int_{B_r(x_0)}\left\{\sum_{i=1}^{m-1}\delta^i\left(\sum_{j=m-i}^{m}\nabla^jA\nabla^{2m-i-j}u\right)+\delta K\right\}\Phi\lesssim\varepsilon\sum_{i=1}^{m}\|\nabla^iu\|_{L^{p_i}(B_r(x_0))}.$$
Combining the above two estimates gives
$$\|\Delta^{\frac{m-2}{2}}v\|_{L^{p_m}(B_r(x_0))}
\leq C\varepsilon\sum_{i=1}^{m}\|\nabla^iu\|_{L^{p_i}(B_r(x_0))}+C\|f\|_{M^{p,p\lambda}(B^{2m})}r^\lambda.$$

\textbf{Part 2.} Estimate of $\|\nabla^{m-2}h\|_{L^{p_m}(B_{r/2}(x_0))}$ 

Since $A\Delta u\in W^{m-1,\frac{2m}{m+1}}(B_r(x_0))$, we can apply \cite[Lemma 5.1]{GXZ2022} to obtain that $h\in W^{m-1,\frac{2m}{m+1}}(B_r(x_0))$ with
\begin{equation*}
	\|h\|_{W^{m-1,\frac{2m}{m+1}}(B_r(x_0))}\leq C\|A\Delta u\|_{W^{m-1,\frac{2m}{m+1}}(B_r(x_0))}\leq C\Gamma(r)r^{\lambda_0},
\end{equation*}
where $\lambda_0$ is defined as in \eqref{eq:lambda_0} and
$$\Gamma(r)=\sum_{i=1}^{m+1}\|\nabla^iu\|_{L^{p_i}(B_r(x_0))}.$$
In particular, this implies that
$$\|h\|_{W^{m-1,\frac{2m}{m+1}}(B_r(x_0))}\leq CMr^{\lambda_0}.$$
where $M$ is defined as in \eqref{eq:M}. It follows from the poly-harmonicity of $h$ and the Sobolev embedding $W^{1,\frac{2m}{m+1}}(B_r(x_0))\subset L^2(B_r(x_0))$ that
\begin{equation*}
	\begin{split}
		\|\nabla^{m-2}h\|_{L^{p_m}(B_{r/2}(x_0))}
		&\leq Cr^{\frac{m(2-p)}{p}}\|\nabla^{m-2}h\|_{L^\infty(B_{r/2}(x_0))}
		\leq Cr^{\frac{m(2-p)}{p}}\left(\vint_{B_r(x_0)}|\nabla^{m-2}h|^2\right)^\frac{1}{2}\\
		&\leq Cr^{\frac{m(2-p)}{p}}r^{-m}\|\nabla^{m-2}h\|_{W^{1,\frac{2m}{m+1}}(B_r(x_0))}
		\leq CMr^{\lambda}.
	\end{split}
\end{equation*}

\textbf{Step 3.} Conclusion.

For the upper bound on $\|\nabla u\|_{L^{p_1}}$, using a standard interpolation argument (see e.g.~\cite[Section 5.1]{Adams-book}) and the H\"older estimate \eqref{eq:holder estimate}, we obtain
\begin{equation*}
	\begin{split}
		\|\nabla u\|_{L^{p_1}(B_{r/4}(x_0))}
		&\leq C\|\nabla^mu\|_{L^{p_m}(B_{r/4}(x_0))}+Cr^{-m}\|u-u_{B_{r/4}(x_0)}\|_{L^{p_m}(B_{r/4}(x_0))}\\
		&\leq C\|\Delta^{\frac{m}{2}}u\|_{L^{p_m}(B_{r/2}(x_0))}+Cr^{-2m(1-1/p)}\|u-u_{B_{r/2}(x_0)}\|_{L^{\infty}(B_{r/2}(x_0))}\\
		&\leq C\left(\|\Delta^{\frac{m-2}{2}}v\|_{L^{p_m}(B_{r/2}(x_0))}+\|\Delta^{\frac{m-2}{2}}h\|_{L^{p_m}(B_{r/2}(x_0))}\right)+CMr^{\lambda}\\
		&\leq C\varepsilon\sum_{i=1}^{m}\|\nabla^iu\|_{L^{p_i}(B_r(x_0))}+CMr^{\lambda}.
	\end{split}
\end{equation*}
The estimates for the terms $\nabla^2u,\nabla^3u,...,\nabla^mu$ are similar and thus it remains to estimate $\|\nabla^{m+1}u\|_{L^{p_{m+1}}}$. By Step 3 in the proof of Theorem 1.2 in \cite{GXZ2022} (more precisely, the last second estimate before equation (5.11) there), we have
\begin{equation*}
	\begin{split}
		\|\nabla^{m+1}u\|_{L^{p_{m+1}}(B_{r/4}(x_0))}
		&\lesssim \sum_{i=1}^{m}\|\nabla^iu\|_{L^{p_i}(B_{r/2}(x_0))}+r^{-m}\|u-u_{B_{r/2}(x_0)}\|_{L^{p_m}(B_{r/2}(x_0))}\\
		&\leq C\varepsilon\sum_{i=1}^{m}\|\nabla^iu\|_{L^{p_i}(B_r(x_0))}+CMr^{\lambda}.
	\end{split}
\end{equation*}
Combining the above two estimates, we conclude that
$$\Gamma(r/4)\leq C\varepsilon\Gamma(r)+CMr^{\lambda}.$$
Applying the Simon's iteration lemma (see \cite[Lemma A7]{ST2013}), we infer that there are $\varepsilon_0$ and $r_0(\lambda,p)>0$ sufficiently small such that for all $r\leq r_0$, we have
$$\Gamma(r)\leq CMr^{\lambda},$$
which gives \eqref{eq:morrey estimate}.
\medskip 

\textbf{Case II}: $m$ is an odd integer.
\medskip 


The proof for this case is completely similar to \cite[Section 5.2.2]{GXZ2022} and thus we only outline the main differences. In the case, $u\in W_{\loc}^{m-1,\eta}$ for any $\eta<p_{m-1}=\frac{2mp}{2m-(m+1)p}$ and $m-1$ is an even integer. Furthermore, we have
\[
\|\nabla^{m-1}u\|_{L^q\big(B_{\frac{r}{2}}(x_0)\big)}\approx\|\Delta^{\frac{m-3}{2}}(A\Delta u)\|_{L^q\big(B_{\frac{r}{2}}(x_0)\big)}.
\]
In this case, we may repeat exactly what we have done in \textbf{Case I}. The only difference is that instead of first showing that $u\in W^{m,p_m}_{\loc}$, we first show that $u\in W^{m-1,p_{m-1}}_{\loc}$ and
$$	\|u\|_{W^{m-1,p_{m-1}}(B_{\frac{r}{2}}(x_0))}\le Cr^{\lambda}\left(\|f\|_{M^{p,p\lambda}(B_{r}(x_0))}+\|u\|_{W^{m-1,2}(B_{r}(x_0))}\right).$$
With a similar interpolation argument as in the previous case, we may conclude that
$$\|u\|_{W^{m,p_m}(B_{\frac{r}{2}}(x_0))}\le Cr^{\lambda}\left(\|f\|_{M^{p,p\lambda}(B_{r}(x_0))}+\|u\|_{W^{m,2}(B_{r}(x_0))}\right).$$
Then once again an interpolation argument leads to
\begin{equation*}
	\|\nabla^{m+1}u\|_{L^{p_{m+1}}\big(B_{\frac{r}{2}}(x_0)\big)}\lesssim \varepsilon\sum_{i=1}^{m}\|\nabla^iu\|_{L^{p_i}(B_r(x_0))}+ r^{\lambda}\left(\|f\|_{M^{p,p\lambda}(B_{r}(x_0))}+\|u\|_{W^{m,2}(B_{r}(x_0))}\right).
\end{equation*}
Finally,  as in the even case dealt above, combining the previous esimate together with Simon's iteration lemma,  gives the desired estimate \eqref{eq:morrey estimate}. The proof of Theorem~\ref{thm:morrey estimate} is complete.
\end{proof}

%

\end{document}